\documentclass[final]{siamltex}
\usepackage[T1]{fontenc}

\usepackage{amssymb,latexsym,amsmath}
\usepackage{mathrsfs}
\usepackage{enumitem}
\usepackage{caption}
\usepackage{comment}
\usepackage{inputenc}
 \usepackage{makeidx}
\usepackage{color}

\allowdisplaybreaks

\newtheorem{assumption}{Assumption}[section]
\newtheorem{example}{Example}[section]

\allowdisplaybreaks

\usepackage{epsfig} 
\usepackage{verbatim}

\usepackage{amsfonts}
\usepackage{float}
\usepackage{multirow}
\title{Zero-Sum Games involving Teams against Teams: Existence of Equilibria, and Comparison and Regularity in Information \thanks{Supported
    by the Natural Sciences and Engineering Research Council of Canada. A preliminary version of this paper was presented at the 2021 IEEE Conference on Decision and Control.}}

\author{Ian Hogeboom-Burr \and Serdar Y\"{u}ksel\thanks{The authors are with the Dept. of Mathematics and Statistics, Queen's University, Kingston K7L 3N6, ON, Canada, {\tt\small \{15ijhb,yuksel\}@queensu.ca}.}}

\begin{document}
\maketitle

\begin{abstract}
Many emerging problems involve teams of agents taking part in a game. Such problems require a stochastic analysis with regard to the correlation structures among the agents belonging to a given team. In the context of Standard Borel spaces, this paper makes the following contributions for two teams of finitely many agents taking part in a zero-sum game: (i) An existence result will be presented for saddle-point equilibria in zero-sum games involving teams against teams when common randomness is assumed to be available in each team with an analysis on conditions for compactness of strategic team measures to be presented. (ii) Blackwell's ordering of information structures is generalized to $n$-player teams with standard Borel spaces, where correlated garbling of information structures is introduced as a key attribute; (iii) building on this result Blackwell's ordering of information structures is established for team-against-team zero-sum game problems. (iv) Finally, continuity of the equilibrium value of team-against-team zero-sum game problems in the space of information structures under total variation is established.
\end{abstract}

\section{Introduction}

In this paper we study zero-sum games where two sets of agents play against one another. For such problems, the informational and correlation properties of agent policies play a crucial role. We study existence of equilibrium solutions, present a characterization of comparison of information structures, and finally establish regularity properties of equilibrium values in information structures.

The notions of statistical models (which we call information structures) and comparison of statistical experiments were characterized in a seminal paper by Blackwell \cite{Blackwell1951}. For finite models, Blackwell's results provide necessary and sufficient criteria for comparing information structures given his notion of a ``more informative'' information structure, where an information structure is better for a decision-maker than another if it ensures the player will never experience a worse expected cost under the former than under the latter, where the cost is a function of a hidden state variable and the player's action, and where the action is only a function (possibly noisy with an independent noise realization) of the player's measurement. Blackwell's results induce a partial order on the set of all information structures and, accordingly, many pairs of information structures are incomparable under his criteria. Le Cam generalized Blackwell's criteria through what is now known as the \textit{Le Cam deficiency}, which allows for the proximity of any two information structures to be defined \cite{LeCamSufficiency}. 

While Blackwell and Le Cam studied information structures in the context of single-player decision problems, information structures can also be studied in the context of multiplayer games. In zero-sum games, Gossner and Mertens worked on the problem of comparing information structures in \cite{gossner2001value}, and this comparison was fully characterized in the finite case by P\k{e}ski in \cite{pkeski2008comparison}; these results were extended to a standard Borel setup in \cite{InformationHY}. This directly extends Blackwell's comparison theorem to zero-sum games. Information structures in team problems (or identical interest games) were studied in a finite-space setting by Lehrer, Rosenberg, and Shmaya \cite{Lehrer}, where Blackwell's theorem was extended to two-player team problems in finite spaces, as well as in \cite[Chapter 4]{YukselBasarBook}.

\subsection{Subtleties of Information Structures in Stochastic Games}

The study of information structures for multi-player setups inevitably necessitates a tedious geometric and topological analysis. Indeed, in the context of team and game problems, non-convexity properties on correlation structures (or strategic measures) and associated non-existence of equilibrium results, (see in particular \cite{anantharam2007common} for non-existence of equilibria and also \cite{CHSH1969, saldiyukselGeoInfoStructure} for convexity properties) and compactness/non-compactness properties \cite{saldiyukselGeoInfoStructure} (see also \cite{gupta2014existence,YukselSaldiSICON17}), require a careful stochastic analysis with many counterexamples and sufficient conditions presented.

In addition to such geometric and topological challenges, in the following we emphasize a number of information structure dependent qualitative attributes of equilibrium solutions in stochastic games. Unlike teams, the dependence on information is rather intricate in stochastic game theory:

\begin{enumerate}
\item More information to a given player does not imply better performance; see e.g. \cite{Bassan2003} for a dynamic model counterexample. For completeness, we present an example with static information, in Example \ref{moreInfoHurts}, where we will see that more information can hurt players. 
\item The team theoretic or stochastic control argument of a player {\it choosing to ignore the additional information as a validation of positive value of information} does not apply in games, even in zero-sum games, due to equilibrium behaviour.
\item As shown in \cite{hogeboom2021continuity}, for general non-zero sum games, equilibrium values are not continuous in information structures even under total variation unlike the setup in team theory.
\end{enumerate}

On the other hand, team theory and zero-sum games have many crucial properties in common: They both admit well-defined values (though with possibly non-unique equilibrium policies) and they both possess strong continuity properties in information, as studied recently in \cite{hogeboom2021continuity} (see also \cite{YukselOptimizationofChannels}). Some general properties of information structures in game theory are studied in detail in \cite{mertens2015repeated,tBasarStochasticDiffGames,witsenhausen1971relations,sanjari2021gameoptimality}. Accordingly, for zero-sum games, several positive attributes regarding informational structure-dependent properties of equilibrium solutions arise:

\begin{enumerate}
\item Under mild conditions (to be reviewed later, e.g. not requiring absolute continuity conditions on information structures), every zero-sum game has a Nash equilibrium solution (also called a saddle-point solution). 
\item A complete theory of ordering of information structures, generalizing that of Blackwell's, is possible; see \cite{pkeski2008comparison} and \cite{InformationHY} for finite and standard Borel setups, respectively. An implication is that, no additional information can hurt a player, though this argument is quite subtle in the game case: As noted, the {\it choose to ignore the additional information and thus no information cannot hurt} argument of stochastic control is not applicable, even though the implication that more information cannot hurt turns out to be correct via a more delicate (geometric) argument.
\item Similar to team problems, equilibrium values for zero-sum games are continuous in information structures under total variation and share similar (though slightly refined) semi-continuity properties under weak/setwise convergences of information structures.
\end{enumerate}

\subsection{Contributions} In this paper in the context of zero-sum games where two sets of agents play against one another, we make the following contributions.

\begin{itemize}
    \item An existence result will be presented for equilibria in zero-sum games involving teams against teams when common randomness will be assumed to be available in Theorem \ref{SPExistTvsT}. As a further technical contribution, a key compactness condition is provided in Theorem \ref{LCcompact} (which will also be used to establish compactness of correlated garblings of information structures).
    \item Blackwell's ordering of information structures will be extended to $n$-player teams with standard Borel spaces in Theorem \ref{Blackwellteam}, as will Le Cam's deficiency results. In this context, correlated garbling of information structures will be introduced as a defining attribute. Our result generalizes \cite{LehrerRosenbergShmaya}, which considered finite models.
    \item Blackwell's ordering of information structures will be generalized to team-against-team zero-sum game problems in Theorem \ref{BlackwellTaT}. In particular, we obtain a partial order on the space of information structures, with an implication being that more information given to a team of decision makers cannot hurt the team, generalizing a corresponding zero-sum game result given in \cite{pkeski2008comparison} and \cite{InformationHY}.
    \item The continuity of value of team-against-team zero-sum game problems in the space of information structures under total variation will be established in Theorem \ref{teamateamtheoremCont}.
\end{itemize}

\section{Model and some Supporting Results}

\subsection{Value of Information can be Negative in Stochastic Games: An Example with Static Information Structure}

Let us give an example on the negative value of information in general stochastic games. While there exist many examples involving dynamic information structures (e.g. \cite{Bassan2003}), the following example shows that even when the information at an agent only depends on exogenous randomness, having more information can hurt an agent in a stochastic game.

\begin{example}\label{moreInfoHurts}
Let $\mathbb{X} = [0,1]$, with prior distribution $\zeta$ defined by the continuous uniform distribution. Let there be two teams of decision makers, each consisting of $N$ individual decision makers, and each with measurement spaces $\mathbb{Y}^i_1= [0, 1]$, $1 \leq i \leq N$ and $\mathbb{Y}^j_2= [0, 1]$ for $1 \leq j \leq N$. We also define the action spaces of the respective DMs as $\mathbb{U}^i_1 = \mathbb{U}^j_2 = [0,1]$ again with $1 \leq i,j \leq N$, and the teams' cost functions as, with ${\bf u}_i = \{u_i^1, \cdots, u_i^N\}$, $i=1,2$:
\begin{equation}
    c_1(x, {\bf u}_1, {\bf u}_2) = \begin{cases} (x - \frac{1}{N} \sum_{j=1}^N u^j_1)^2 + 2, & \frac{1}{N} \sum_{j=1}^N u^j_2 = x \\
    2(x - \frac{1}{N} \sum_{j=1}^N u^j_1)^2 + (\frac{1}{N} \sum_{j=1}^N u^j_1 - \frac{1}{N} \sum_{j=1}^N u^j_2)^2,  & \frac{1}{N} \sum_{j=1}^N u^j_2 \neq x \nonumber
    \end{cases}
\end{equation}

\begin{equation}
    c_2(x, {\bf u}_1, {\bf u}_2) = \begin{cases} (x - \frac{1}{N} \sum_{j=1}^N u^j_2)^2 + 2, & \frac{1}{N} \sum_{j=1}^N u^j_1 = x \\
    2(x - \frac{1}{N} \sum_{j=1}^N u^j_2)^2 + (\frac{1}{N} \sum_{j=1}^N u^j_1 - \frac{1}{N} \sum_{j=1}^N u^j_2)^2,  & \frac{1}{N} \sum_{j=1}^N u^j_1 \neq x \nonumber
    \end{cases}
\end{equation}

Under measurement channels where neither of the team has any information regarding $x$, (which could be achieved for instance by measurement channels returning $y^i_m, m=1,2; 1\leq i\leq N$ which are independent of $x$, via a channel with zero information theoretic capacity), there exists an essentially unique Nash equilibrium solution given with $\frac{1}{N} \sum_{i=1}^N u^i_1 = \frac{1}{2} = \frac{1}{N} \sum_{i=1}^N u^i_2$. This uniqueness can be seen as follows: when the DMs have no information regarding $x$, the event that $\frac{1}{N} \sum_{i=1}^N u^i_1=x$ or $\frac{1}{N} \sum_{i=1}^N u^i_2=x$ has zero measure and therefore the complementary conditions are active for defining the cost in the above. Note though that in this case the cost functions essentially turn the problem into a team problem since minimization over $\frac{1}{N} \sum_{i=1}^N u^i_1$ of 
\[2(x - \frac{1}{N} \sum_{i=1}^N u^i_1)^2 + (\frac{1}{N} \sum_{i=1}^N u^i_1 - \frac{1}{N} \sum_{i=1}^N u^i_2)^2\]
will lead to the same solution as the minimization of
\[2(x - \frac{1}{N} \sum_{i=1}^N u^i_1)^2 + (\frac{1}{N} \sum_{i=1}^N u^i_1 - \frac{1}{N} \sum_{i=1}^N u^i_2)^2 + 2(x - \frac{1}{N} \sum_{i=1}^N u^i_2)^2.\]
The same applies for $\frac{1}{N} \sum_{i=1}^N u^i_2$, and so we can view the DMs as solving essentially the same problem. Therefore, we have a standard static quadratic team problem which admits an essentially unique optimal solution (where the cost is strictly convex in the average team action, though not on the collection of actions) \cite[Theorem 2.6.3]{YukselBasarBook} with $\frac{1}{N} \sum_{i=1}^N u^i_1=\frac{1}{2}=\frac{1}{N} \sum_{i=1}^N u^i_2$. This leads to an expected cost of $\frac{1}{6}$ for each team.

However, when the DM measurement channels are such that all DMs have perfect information regarding $x$ (e.g. through $y^i_1 = x  = y^j_2, \qquad 1 \leq i,j, \leq N$), we first see that another equilibrium solution arises: with $\frac{1}{N} \sum_{i=1}^N u^i_1 = x = \frac{1}{N} \sum_{i=1}^N u^i_2$. One can observe this is a Nash equilibrium by noting that if DM $2$ holds their strategy constant as $\frac{1}{N} \sum_{i=1}^N u^i_2 = x$, then DM 1 will play the game with cost function $(x - \frac{1}{N} \sum_{i=1}^N u^i_1)^2 + 2$, which is minimized by playing $\frac{1}{N} \sum_{i=1}^N u^i_1 =x$. The same holds true for DM 2 in the reverse case where DM 1's strategy is fixed. We now make the point that this is the essentially unique deterministic Nash equilibrium:

%
If the condition $\frac{1}{N} \sum_{i=1}^N u^i_1=x=\frac{1}{N} \sum_{i=1}^N u^i_2$ is not active, we would again reduce to a team problem and in this case, we would have \[\gamma^i_1(x) = \frac{2}{3}x+\frac{1}{3} \frac{1}{N} \sum_{i=1}^N \gamma^i_2(x), \qquad \gamma^i_2(x) = \frac{2}{3}x+\frac{1}{3} \frac{1}{N} \sum_{i=1}^N \gamma^i_1(x)\]
leading to, by (essential) uniqueness \cite[Theorem 2.6.3]{YukselBasarBook}, $ \frac{1}{N} \sum_{i=1}^N \gamma^i_1(x)=x= \frac{1}{N} \sum_{i=1}^N \gamma^i_2(x)$, which however would take us back to the inefficient equilibrium given above. Thus, the (essentially) unique pure strategy Nash equilibrium for the full-information problem is at $ \frac{1}{N} \sum_{i=1}^N u^i_1 = x = \frac{1}{N} \sum_{i=1}^N u^i_2$. This results in an expected cost of $2$ for each of the 2 teams of DMs.

This example demonstrates a static team-against-team game in which more information hurts all the DMs.
We note that by adding individual costs such as $(u^i_1)^2$ so that the costs are strictly convex, the essential uniqueness can be strengthened to uniqueness (however the goal is just to show that more information can be detrimental in game problems). For example, if the costs are modified as follows, then uniqueness under the former information structure can be attained and a similar analysis is applicable for the latter if a strict convexity is imposed on the costs in the first cases of both functions:
\begin{equation}
    c^1(x, {\bf u}^1, {\bf u}^2) = \begin{cases} (x - \frac{1}{N} \sum_{i=1}^N u^i_1)^2 + 2, & \frac{1}{N} \sum_{i=1}^N u^i_2 = x \\
    2(\frac{3}{2} x - \frac{1}{N} \sum_{i=1}^N u^i_1)^2 + \frac{1}{N} \sum_{i=1}^N (u^i_1)^2 + (\frac{1}{N} \sum_{i=1}^N u^i_1 - \frac{1}{N} \sum_{i=1}^N u^i_2)^2,  & \frac{1}{N} \sum_{i=1}^N u^i_2 \neq x \nonumber
    \end{cases}
\end{equation}
\begin{equation}
    c^2(x, {\bf u}^1, {\bf u}^2) = \begin{cases} (x - \frac{1}{N} \sum_{i=1}^N u^i_2)^2 + 2, & \frac{1}{N} \sum_{i=1}^N u^i_1 = x \\
    2(\frac{3}{2} x - \frac{1}{N} \sum_{i=1}^N u^i_2)^2 + \frac{1}{N} \sum_{i=1}^N (u^i_2)^2 + (\frac{1}{N} \sum_{i=1}^N u^i_1 - \frac{1}{N} \sum_{i=1}^N u^i_2)^2,  & \frac{1}{N} \sum_{i=1}^N u^i_1 \neq x \nonumber
    \end{cases}
\end{equation}

Note that in this case, $u^j_1 = \frac{1}{2}$ is the unique solution in the no information case and $u^i_m=x, m=1,2; 1 \leq i \leq N$ is the essentially unique solution in the full information case.

\end{example}

We will see, however, that in the context of zero-sum games involving teams of DMs, a partial ordering is possible and more information cannot hurt a DM (or team).

\subsection{Stochastic Teams and a Supporting Result}
Consider $n \in \mathbb{Z}_{\geq 1}$ players on a team in a single-stage game. Let $\Omega_0$ denote a standard Borel state space, with $x \sim \zeta$ an $\Omega_0$-valued random variable known as the state of nature. Recall that a standard Borel space is a Borel subset of a complete, separable, metric (Polish) space. Let $\zeta$ be the {\it prior} probability distribution on a hidden state, which is common knowledge to all players. We denote by $\mathbb{Y}^i$ Player $i$'s standard Borel measurement space, for $i \in \{1, \dots, n\}$. We use $-i$ to denote all players other than Player $i$. For each player, we define their measurement $y^i$, which is a $\mathbb{Y}^i$-valued random variable, where: $y^i = g^i(x, v^i)$ for some noise variable $v^i$ (and which, without any loss, can be taken to be $[0,1]$-valued). By stochastic realization arguments \cite{BorkarRealization}, the above formulation is equivalent to viewing $y^i$ as being generated by a measurement channel $Q^i$, which is a Markov transition kernel from $\Omega_0$ to $\mathbb{Y}^i$.
We denote the joint probability measure on $\Omega_0 \times \mathbb{Y}^1 \times \dots \times \mathbb{Y}^n$ by $\mu(d\omega_0, dy^1, \dots, dy^n)$, and define this as the \textit{information structure}.

For $P\in  \mathcal{P}(\Omega_0)$ and kernel $Q$, we let $P Q$ denote the joint distribution induced on
$(\Omega_0\times \mathbb{Y}, \mathcal{B}(\Omega_0\times
\mathbb{Y}))$ by channel $Q$ with input distribution $P$:
\[  P Q(A) = \int_{A} Q(dy|x)P (d\omega_0), \quad A\in  \mathcal{B}(\Omega_0 \times \mathbb{Y}). \]

Each player also has a standard Borel action space $\mathbb{U}^i$. The team has a common cost function $c: \Omega_0 \times \mathbb{U}^1 \times \dots \times \mathbb{U}^n \rightarrow \mathbb{R}$. 

For fixed prior, state space, and measurement spaces $\zeta$, $\Omega_0$, $\mathbb{Y}^1, \dots, \mathbb{Y}^n$, a \textit{game} for the team is a $n+1$-tuple consisting of a measurable and bounded cost function and an action space for each player, $G = (c, \mathbb{U}^1, \dots, \mathbb{U}^n)$. 

The team's goal is to minimize the expected cost functional for a given game $G$:
\begin{align*}
J(G, \mu, \gamma^1, \dots, \gamma^n) := E^{\mu, \bar{\gamma}}[c(x,\gamma^1(y^1), \dots, \gamma^n(y^n))]
\end{align*}
where the players select their policies from the set of all admissible \textit{policies} $\Gamma^i := \{\gamma: \mathbb{Y}^i \rightarrow \mathbb{U}^i\}$, which are measurable functions from a player's measurement space to their action space. We refer to $u^i = \gamma^i(y^i)$ as the action of the player, and $\gamma^i$ as their policy.  

In team problems, equilibrium solutions are team policies $\bar{\gamma}^* := (\gamma^{1,*}, \dots, \gamma^{n,*})$ which minimize the value function. The value function at equilibrium is denoted by $J^*(c, \mu)$. General existence results for team-optimal policies can be found in \cite[Section 5]{YukselWitsenStandardArXiv}.

To ensure existence of team-optimal policies, throughout this paper we will assume the player action spaces $\mathbb{U}^i$ are compact and the cost function is continuous and bounded in actions, but not necessarily in $\omega_0$. 

We also recall the notion of a strategic measure, which is a probability measure on the state, measurement, and action spaces induced by a policy. Given a fixed team policy $\bar{\gamma} = (\gamma^1, \dots, \gamma^n)$, a \textit{strategic measure} is the joint probability measure induced by $\bar{\gamma}$ on $\Omega_0 \times \prod^n_{k = 1} (\mathbb{Y}^k \times \mathbb{U}^k)$. We now define and revisit various spaces of strategic measures. 

We define $L_C(\mu)$, as in \cite{YukselSaldiSICON17}:
\[L_C(\mu) = \left\{P \in \mathcal{P}\left(\Omega_0 \times \prod_{k=1}^{n}(\mathbb{Y}^k \times \mathbb{U}^k)\right) : P(B) = \int \eta(dz) L_A(\mu, \tilde{\gamma}(z))(B), \eta \in \mathcal{P}(\mathcal{W})\right\}.\]

Where $\mathcal{W} = [0,1]^n$, $\tilde{\gamma}$ represents a collection of team policies measurably parameterized by $z \in \mathcal{W}$ so that the map $L_A(\mu, \tilde{\gamma}(\cdot)):\mathcal{W} \rightarrow L_A(\mu)$ is Borel measurable, and where $L_A(\mu)$ is:
\begin{align*}L_A(\mu) = &\bigg\{P \in \mathcal{P}\left(\Omega_0 \times \prod_{k=1}^{n}(\mathbb{Y}^k \times \mathbb{U}^k)\right) \\ &: P(B) = \int_{B^0 \times \prod_k A^k}\mu(d\omega_0, d\textbf{y})\prod_{k}1_{u_k = \gamma^k(y^k) \in B^k}, \gamma^k \in \Gamma^k, B \in \mathcal{B}\left(\Omega_0 \times \prod_{k=1}^{n}(\mathbb{Y}^k \times \mathbb{U}^k)\right)\bigg\}.
\end{align*}

Note that $L_A(\mu)$ is Borel measurable under the weak convergence topology \cite{saldiyukselGeoInfoStructure}, and we use $L_A(\mu, \bar{\gamma})$ to denote the strategic measure induced by a particular team strategy $\bar{\gamma}$. We define $L_{CCR}(\mu)$ as:

\begin{eqnarray} 
L_{CCR}(\mu) &:=& \bigg\{P \in \mathcal{P}\left(\Omega_0 \times \prod_{k=1}^{n}(\mathbb{Y}^k \times \mathbb{U}^k)\right)  \nonumber \\ 
&& \qquad :  P(B) = \int_{B \times \mathcal{W}}\eta(dz)\mu(d\omega_0, d\bar{y}) \prod_k \Pi^k(du^k|y^k, z), \eta \in \mathcal{P}(\mathcal{W})\bigg\}, \nonumber
\end{eqnarray}
where $\Pi^k$ is a stochastic kernel from $(\mathbb{Y}^k \times \mathcal{W})$ to $\mathbb{U}^k$, for $k = 1, \dots, n$. It was shown in \cite[Theorem 3.1]{saldiyukselGeoInfoStructure} that $L_C(\mu) = L_{CCR}(\mu)$. 

We will also define $L_R(\mu)$ as:
\[L_R(\mu) = \left\{P \in \mathcal{P}\left(\Omega_0 \times \prod_{k=1}^{n}(\mathbb{Y}^k \times \mathbb{U}^k)\right) : P(B) = \int_B \mu(d\omega_0, d\bar{y}) \prod_{k=1}^n \kappa^k(du^k|y^k)\right\},\]
where $\kappa^k$ is taken from the set of stochastic kernels from $\mathbb{Y}^k$ to $\mathbb{U}^k$ for each $k  =1,\dots, n$. 

Essentially, $L_A(\mu)$ is the set of all strategic measures induced by admissible measurable policies, $L_R(\mu)$ is the set induced by allowing independently randomized admissible policies for each player, and $L_C(\mu)$ is the set induced by allowing randomized admissible policies with potentially common randomness among players. 

It was shown in \cite{YukselSaldiSICON17} that $L_C(\mu)$ is convex, which will be crucial in the extension of Blackwell's theorem. Closedness and compactness of $L_C(\mu)$ will be required for some of the results presented here. In general, $L_C(\mu)$ is not closed under the weak convergence topology \cite[Theorem 4.3]{saldiyukselGeoInfoStructure}. However, we can relate sufficient conditions for the compactness of the space $L_R(\mu)$ to $L_C(\mu)$. 

For the following result, we assume that the measurements of the players are either independent or by a change of measure argument, under an absolute continuity condition, have been turned into a static information structure with independent measurements setup, see \cite[Section 2.2]{YukselWitsenStandardArXiv} for details. In particular, we assume that 
\begin{assumption}\label{AbsoluteContIndRedG}
\begin{eqnarray}\label{absContCont1}
P(d\omega_0,dy^1,\cdots,dy^n) \ll \mu_0(d\omega_0) \prod_{i=1}^n \bar{Q}^i(dy^i)
\end{eqnarray}
for some reference measures $\mu_0$ and $\bar{Q}^i$ so that the Radon-Nikodym derivative
\[\frac{d (P(d\omega_0,dy^1,\cdots,dy^n))}{d (\mu_0(d\omega_0) \prod_{i=1}^n \bar{Q}^i(dy^i)))} =: f(\omega_0,y^1,\cdots,y^n)\]
exists. 
\end{assumption}
This condition is important in establishing that $L_R(\mu)$ is compact under the weak convergence topology, as we make formal in the following.

\begin{theorem} \label{existenceRelaxed3} \cite[Theorem 5.2]{YukselWitsenStandardArXiv}
Consider a static or a dynamic team that satisfies Assumption \ref{AbsoluteContIndRedG} (and thus, admits an independent static reduction). The set $L_R(\mu)$ is closed under the weak convergence topology. Furthermore, if the action sets are compact (or under a tightness condition), the set $L_R(\mu)$ is compact under the weak convergence topology.
\end{theorem}


We will now state the following theorem. This result will be critical both for existence of equilibria and for comparison of information structures. 
\begin{theorem}\label{LCcompact}
If $L_R(\mu)$ is compact, then $L_C(\mu)$ is compact. 
\end{theorem}

\begin{proof}
Let $\{\phi_k\}$ be an arbitrary sequence of information structures in $L_C(\mu)$ defined with $B \in \mathcal{B}\left(\Omega_0 \times \prod_{k=1}^{n}(\mathbb{Y}^k \times \mathbb{U}^k)\right)$ as

\[\phi_k(B) = \int_{B} \left( \int \nu_k(d\theta) \theta(dz)\right).\]
Following the relationship between $L_{CCR}(\mu)$ and $L_R(\mu)$, as well as \cite[Theorem 3.1]{saldiyukselGeoInfoStructure}, we can view $L_C(\mu)$ as the mixture of elements of $L_R$. The convergence properties of the sequence $\{\phi_k\}$ can be studied by the convergence of the family of measures:
\[\phi_k(B) = \int_{B} f_m(z) \left( \int \nu_k(d\theta) \theta(dz)\right),\]
for a fixed countable collection of {\it weak-convergence determining} functions $f_m:\Omega \times (\prod_{i=1}^n \mathbb{X}^i \times \mathbb{Y}^i) \rightarrow \mathbb{R}$ \cite[Theorem 3.4.5]{ethier2009markov}. 

Now, using Fubini's Theorem, we can rewrite the above as:
\[\int \nu_k(d\theta) \int f_m(z) \theta(dz),\]
and can note that the inner integral is continuous and bounded in $\theta$.

Thus we can write this as: \[\int \nu_k(d\theta) F_m(\theta),\] for the continuous and bounded $F_m$ on $L_R(\mu)$ given with $F_m(\theta) := \int f_m(z) \theta(dz)$. 

Since $L_R(\mu)$ is weakly compact, the set of probability measures on ${\cal B}(L_R(\mu))$ is weakly compact \cite{Bil99}; and accordingly there exists a subsequence of $\{\nu_k\}$, $\{\nu_{k_l}\}$, which converges weakly to some $\nu$ in ${\cal P}(L_R(\mu))$. Since $F_m$ is continuous and bounded on $L_R(\mu)$, it follows then that $\int \nu_{k_l}(d\theta) F_m(\theta)$ will converge to $\int \nu(d\theta) F_m(\theta)$. This applies for each $m$, and therefore, we conclude that $\{\phi_k\}$ has a convergent subsequence in $L_C(\mu)$. Thus, $L_C(\mu)$ is compact. 
\end{proof}





This gives us sufficient conditions to ensure compactness of $L_C(\mu)$. This result will be crucial in our analysis both for the existence of equilibria as well as on a key closedness property involving a set of correlated garblings in a team of agents.

\subsection{Team-against-Team Zero-Sum Games}\label{InfoStTTG1}
For teams against teams, we consider two teams. Team 1 (the minimizing team) consists of $N_1$ DMs and Team 2 (the maximizing team) consists of $N_2$ DMs. The two teams compete in a single-stage zero-sum game with private information. Here we use a subscript $i$ to denote the team to which a DM belongs. 

The teams share a common cost function $c: \Omega_0 \times \mathbb{U}^1_1 \times \dots \times \mathbb{U}^{N_1}_1 \times \mathbb{U}^1_2 \times \dots \times \mathbb{U}^{N_2}_2 \rightarrow \mathbb{R}$. Team 1's goal is to minimize the expected cost and Team 2's goal is to maximize it. 

The information structure for the game $\mu$ is a joint probability measure on $\Omega_0 \times \mathbb{Y}^1_1 \times \dots \times \mathbb{Y}^{N_1}_1 \times \mathbb{Y}^1_2 \times \dots \times \mathbb{Y}^{N_2}_2$. We use $\mu_i$ to denote the marginal of $\mu$ on $\Omega_0 \times \mathbb{Y}^i_i \times \dots \times  \mathbb{Y}^{N_i}_i$ for $i = 1,2$ (i.e. $\mu_i$ captures the information available to team $i$). 

For teams-against-teams we use the notation $V$ and $V^*$, rather than $J$ and $J^*$, to denote the value and equilibrium value of a game, respectively. 

\begin{definition} Given an information structure $\mu$, let ${\bf \Gamma}_i$ denote the set of team policies of Team $i$, $i=1,2$. We say that $\bar{\gamma}^{*}_1, \bar{\gamma}^{*}_2$ is an equilibrium for a team-against-team game if
\begin{equation}\inf_{\bar{\gamma}_1 \in {\bf \Gamma}_1} V(\mu,\bar{\gamma}_{1},\bar{\gamma}^{*}_2) = V(\mu,\bar{\gamma}^{*}_1,\gamma^{*}_2) = \sup_{\bar{\gamma}_2 \in {\bf \Gamma}_2}  V(\mu,\gamma^{*}_1,\bar{\gamma}_2). \nonumber \end{equation}
\end{definition}

\section{Team-against-Team Zero-Sum Games: Existence of Saddle Point Equilibria}

First, we establish an existence result for saddle-point equilibria. As emphasized explicitly in \cite{anantharam2007common}, lack of convexity in the set of team policies can lead to games without a saddle-point. We will allow each team to select a team policy in which players may share common randomness. This randomness is independent between teams. Allowing this, rather than only independently randomized policies for each team, ensures that the team policy spaces are convex, which is crucial in demonstrating the existence of an equilibrium.

\cite[Theorem 3.1]{balder1988generalized} (and \cite[Theorem 3.1]{InformationHY}) study existence of saddle-point equilibria in two-player zero-sum games under an absolute continuity condition and these have been relaxed in \cite{mamer1986zero} (also \cite[Theorem 3.4]{balder1988generalized}, \cite[Theorem 3.2]{InformationHY}). In the following, we present a generalization of the relaxed version where absolute continuity among the different team-players is not required (thus Assumption \ref{ABSContA1} presented below does not necessarily apply, but does apply within the team under Assumption \ref{ABSContA3} below). 

For $i = 1,2$, for fixed information structure $\mu$, we define $\mu_i$ as the the marginal of $\mu$ on $\Omega_0 \times \prod_{j=1}^{N_i} \mathbb{Y}^j_i$. $\mu_i$ represents the portion of the information structure $\mu$ that is relevant for Team $i$'s action selection (as they can not see the opposing team's measurements). We then we define $L_C(\mu_i)$ following our usual definition for a single team, allowing for actions to be selected using independent (but possibly shared between players) randomness. 
 
 We will assume that $L_C(\mu_i)$ is closed for every fixed policy of the opponent team, a sufficient condition for this is given below in Assumption \ref{ABSContA3}.
 
\begin{assumption}\label{ABSContA3}
Equation (\ref{absContCont1}) holds for both $\mu_1$ and $\mu_2$.
\end{assumption}


\begin{theorem}\label{SPExistTvsT}
For a given team-against-team game, assume the following hold.
\begin{enumerate}[nosep]
    \item[(i)] Team policies are selected such that the strategic measures belong to $L_C(\mu_1), L_C(\mu_2)$, respectively. 
    \item[(ii)] Assumption \ref{ABSContA3} applies (the implication, to be shown in the proof, being that $L_C(\mu_1)$ (when Team 2's policy is fixed) and $L_C(\mu_2)$ (when Team 1's policy is fixed) are compact under weak convergence). 
    \item[(iii)] The cost function $c$ is bounded and continuous in players' actions for any $\omega_0 \in \Omega_0$.
\end{enumerate}    
Then an equilibrium exists. 
\end{theorem}
\begin{proof}
In what follows, we always assume a team policy $\bar{\gamma}_i = (\gamma^1_i, \dots, \gamma^{N_i}_i)$ is such that $u^k_i = \gamma^k_i(y^k_i, z_i)$, for some independent randomness $z_i \in [0,1]^{N_i}$ for each team, for $k = 1, \dots, N_i$, $i = 1,2$. I.e., policies are selected such that the induced strategic measures are in $L_C(\mu_1)$ and $L_C(\mu_2)$. Note that the measure on $(\Omega \times \prod_{j=1}^{N_1}\mathbb{Y}^j_1 \times \prod_{k=1}^{N_2}\mathbb{Y}^k_2)$ is fixed for both teams by the information structure $\mu$, and so $L_C(\mu_i)$ has a fixed marginal on the state and measurement spaces.

Let $L_C(\mu_i, \bar{\gamma}_i)$ be team $i$'s strategic measure induced by team policy $\bar{\gamma}_i$, for $i = 1,2$. Denote by $S_i$ the projection of the space $L_C(\mu_i)$ onto $\prod_{k=1}^{N_i}(\mathbb{Y}^{k}_{i} \times \mathbb{U}^{k}_{i})$, eliminating the state space from the measure and simplifying notation for convenience. Since the measure on the exogenous variables is fixed by the information structure $\mu$, $S_i$ will inherit convexity (as conditional independence will be preserved under convex combinations given that the opponent team policy is fixed) and weak compactness from $L_C(\mu_i)$ (following from the analysis in the proof of \cite[Theorem 3.2]{InformationHY}); see \cite[Theorem 5.2]{saldiyukselGeoInfoStructure} which shows that conditional independence is preserved when only one player perturbs the strategic measure (unlike the case where multiple players make simultaneous perturbations in \cite[Theorem 4.2]{saldiyukselGeoInfoStructure}). This applies in this context by viewing each team as a large single-agent.


Now, without loss of generality, we will fix Team 2's policy as $\bar{\gamma}_2^*$ (which is an arbitrary policy that induces a strategic measure $s_2^*$ in $L_C(\mu_2)$). We then have:
\begin{align}
  &V_{G}^{\mu}(s_1, s_2^*) \nonumber\\
   &= \int s_1 (dy^1_1, \dots, dy^{N_1}_1, du^1_1, \dots, du^{N_1}_1) \dots \\ & \quad \quad \quad \quad  \left(\int \mu^{{\bar{\gamma}_2^*}}(d\omega_0,dy^1_2, \dots, dy^{N_2}_2, du^1_2, \dots, du^{N_2}_2 | y^1_1, \dots, y^{N_1}_1)c(\omega_0,u^1_1,\dots, u_1^{N_1}, u_2^1, \dots, u_2^{N_1})\right), \label{DeriveCondIndG}
\end{align} 
where $\mu^{{\bar{\gamma}_2}^*}$ denotes the strategic measure induced by Team 2's policy ${\bar{\gamma}_2}^*$. Define the bracketed term as $\bar{c}( y^1_1, \dots, y_1^{N_1}, u^1_1, \dots, u_1^{N_1})$. We can observe that, by assumption, $\bar{c}$ is continuous in $u^1_1, \dots, u_1^{N_1}$, and furthermore it is measurable in $y^1_1, \dots, y_1^{N_1}$. 

Recall the $w$-$s$ topology \cite{schal1975dynamic} on the set of probability measures ${\cal P}(\Omega_0 \times \mathbb{U})$; this is the coarsest topology under which $\int f(x,u) \nu(d\omega_0,du): {\cal P}(\Omega_0 \times \mathbb{U}) \to \mathbb{R}$ is continuous for every measurable and bounded $f$ which is continuous in $u$ for every $x$ (but unlike weak topology, $f$ does not need to be continuous in $x$). Now, since the exogenous variables are fixed, weak convergence in this setting is equivalent to $w$-$s$ convergence, and continuity in the exogenous variable is not needed here: Consider a sequence of measures $\{s_1\}_k$ which converges to $s_1$ weakly. We have that $\Bar{c}(u^1_1, \dots, u^1_{N_1}, y^1, \dots, y^1_{N_1})$ is continuous in $u^1_1, \dots, u^1_{N_1}$ by assumption. Since $\mu$ is fixed, the marginals on $\prod_{k = 1}^{N_2}\mathbb{Y}^2_k$ are fixed. Therefore, by \cite[Theorem 3.10]{schal1975dynamic} (or \cite[Theorem 2.5]{balder2001} and \cite{YukselWitsenStandardArXiv}), we can use the $w$-$s$ topology on the set of probability measures $\mathcal{P}(\prod_{k=1}^{N_2}\mathbb{Y}_2^k \times \mathbb{U}_2^k)$. And so we have continuity of $V^{\mu}_G$ in $s_1$ in the $w$-$s$ topology and, by the equivalence in this setting, the weak topology. 

This also holds for continuity in $s_2$ in the reverse case were we fix Team 1's strategic measure as an arbitrary $s_1^*$. 

Following compactness and convexity of $S_1$ and $S_2$, by the continuity (and linearity in each entry) of $V_{G}^{\mu}(s_1, s_2)$, by Fan's minimax theorem, \cite[Theorem 1]{Fan}, we have that an equilibrium exists:
\begin{equation}
    \min_{S_1}\max_{S_2} V_{G}^{\mu}(s_1, s_2) = \max_{S_2}\min_{S_1}V_{G}^{\mu}(s_1, s_2)\nonumber
\end{equation} 

\end{proof}

The setup where absolute continuity applies also between the teams would be a special case, which would make the proof presentation slightly more direct, as we note below.

\begin{assumption}\label{ABSContA1}
 \[\mu(d\omega_0;dy^1_1, \dots, dy^{N_1}_2;dy^1_2, \dots, dy^{N_2}_2) \ll \zeta(d\omega_0)\mu_1(dy^1_1, \dots, dy^{N_1}_2)\mu_2(dy^1_2, \dots, dy^{N_2}_2).\]
 \end{assumption}

In Theorem \ref{SPExistTvsT}, if we had also imposed Assumption \ref{ABSContA1}, the analysis following (\ref{DeriveCondIndG}) would be more direct. However, strictly speaking, this assumption is not needed as absolute continuity across the teams is not required (only within members of a team). 

\section{Comparison of Information Structures in Team-against-Team Zero-Sum Games}

We now establish a comparison result on information structures for stochastic teams. Before we do this, we present a brief review.

\subsection{Brief Review on Comparison of Information Structures}

We now recall the fundamental results of Blackwell and Le Cam in the comparison of information structures for single-player decision problems. First, we recall the standard notion of garbling.

\begin{definition}\label{garbling}
An information structure induced by some channel $Q_2$ is garbled (or stochastically degraded) with respect to another one, $Q_1$, if there exists a channel $Q'$ on $\mathbb{Y}\times \mathbb{Y}$ such that
\[
Q_2(B|x) = \int_{\mathbb{Y}} Q'(B|y) Q_1(dy|x), \; B \in {\cal
  B}(\mathbb{Y}),\;P \; a.s. \;x \in \mathbb{X}.
\]
\end{definition}
We also define the following notion:
\begin{definition}\label{moreinformative}
An information structure $\mu$ is \textit{more informative than} another information structure $\nu$ if 
\begin{equation}
    \inf_{\gamma \in \Gamma} E^{\nu,\gamma}_P[c(x,u)] \geq \inf_{\gamma \in \Gamma} \nonumber
E^{\mu,\gamma}_P[c(x,u)]
\end{equation}
for all single player decision problems $(c(x,u), \mathbb{U})$. 
\end{definition}

We now recall Blackwell's classical result:

\begin{theorem} \label{BlackwellInformative} [Blackwell \cite{Blackwell}]
Let $\mathbb{X}, \mathbb{Y}$ be finite spaces. The following are equivalent:
\begin{itemize}
\item[(i)]
$Q_2$ is weakly stochastically degraded with respect to $Q_1$ (that is, a garbling of $Q_1$).
\item[(ii)] The information structure induced by channel $Q_1$ is more informative than the one induced by channel $Q_2$ for all single player decision problems with finite $\mathbb{U}$.
\end{itemize}
\end{theorem}

The implication (i)$\rightarrow$(ii) is direct and is applicable to general spaces, as noted in Theorem 4.3.2 of \cite{YukselBasarBook}; the converse direction is more involved: the finite case was studied by Blackwell, the zero-sum games with finite spaces was studied in \cite{pkeski2008comparison}, and the standard Borel case was studied in \cite{InformationHY}.

Le Cam developed a further extension of Blackwell's theorem, and defined the notions of Le Cam deficiency and Le Cam distance. 

\begin{definition}
The Le Cam deficiency of an information structure $\mu$ with respect to another information structure $\nu$ is
$\delta(\mu, \nu) =: \inf_{\kappa \in K}\|\kappa \mu - \nu\|_{TV}$
\end{definition}

We now recall the following theorem due to Le Cam, which holds in the standard Borel setup:

\begin{theorem}[Le Cam \cite{LeCamSufficiency}]\label{LeCam}
Let $\epsilon > 0$ be fixed, $\delta(\mu, \nu) < \epsilon$ if and only if for any policy $\gamma_{\nu}$ under $\nu$ and any game with a bounded cost function $\|c\|_{\infty} \leq 1$, there exists a policy $\bar{\gamma}_{\mu}$ under $\mu$ such that
\begin{equation}
     E^{\mu}_{P}[c(x, \bar{\gamma}_{\mu}(y))]< E^{\nu}_{P}[c(x, \gamma_{\nu}(y))] + \epsilon \nonumber
\end{equation}
\end{theorem}
We also have the following definition from Le Cam:
\begin{definition}
The Le Cam distance between two information structures $\mu$ and $\nu$ is
\begin{equation}
    \Delta(\mu, \nu) := \max\{\delta(\mu,\nu), \delta(\nu,\mu)\}. \nonumber
\end{equation}
\end{definition}
Le Cam's results generalize those of Blackwell, and can be used to define a topology on information structures based on difference in value over all games; this was done for zero-sum games in \cite{pkeski2019value}. 

\subsection{Comparison of Information Structures in Stochastic Teams}
Now, we move to extending Blackwell's results on comparison of information structures for single-player decision problems to team problems. This was done for two-player teams by Lehrer, Rosenberg, and Shmaya in \cite{LehrerRosenbergShmaya} for problems with finite spaces, see also \cite[Chapter 4]{YukselBasarBook}. Here, these results are extended to arbitrary $n$-player teams with standard Borel spaces.

We define the notion of garbling with potentially common randomness for a team, generalizing the notion from Definition \ref{garbling}:

\begin{definition}\label{commonGarbleDefinition}
An information structure $\nu$ is a correlated garbing of an information structure $\mu$ with potentially common randomness if there exists $\eta \in \mathcal{P}(\mathcal{W})$, where $\mathcal{W} = [0,1]^n$, such that:
\[\nu(B) = \int_{B \times \mathcal{W}} \eta(dz)\mu(d\omega_0, d\bar{y}) \prod_k \Pi^k(d\tilde{y}^k|y^k, z), \quad \forall B \in \mathcal{B}(\Omega \times \prod_{i = 1}^n \mathbb{Y}^i)\]
for some stochastic kernels $\Pi^k$ from $(\mathbb{Y}^k \times \mathcal{W})$ to $\mathbb{Y}^k$, for $k = 1, \dots, n$.  
\end{definition}
Note that here $\eta$ is an arbitrary probability measure, as in the definition of $L_C(\mu)$. Assume $L_C(\mu)$ is closed (see Theorem \ref{LCcompact}, a sufficient condition being the independent static reducibility under Assumption \ref{AbsoluteContIndRedG}. 

Let $\mathcal{E}$ denote the set of all team problems with compact action spaces and cost functions that are continuous in player actions for every $\omega_0$. 

\begin{theorem}\label{Blackwellteam}
Let Assumption \ref{AbsoluteContIndRedG} hold so that $L_C(\mu)$ is closed. For team problems with static reductions, an information structure $\mu$ is more informative than an information structure $\nu$ (for all games in $\mathcal{E})$ if only if $\nu$ is a garbling of $\mu$ for each player, with potentially common randomness among the garblings. 
\end{theorem}
\begin{proof}(Sketch)
\textbf{Part 1:} First we study the forward direction. 
Let $\mu$ be an information structure, and let $\nu$ be another information structure such that the measurements under $\nu$ are given by $y^i_{\nu} = h^i(y^i_\mu, Z)$ for $i = 1, \dots, n$ and some independent randomness $Z$, that may be common between players. 

We can clearly observe that any outcome for a game achievable under $\nu$ is achievable under $\mu$, since outcomes are determined by policies, which are measurable functions from measurements to actions, and the measurements under $\nu$ can be simulated using the measurements from $\mu$ and the independent randomness $Z$, and thus the resulting expected outcomes can also be achieved under $\mu$. 

\textbf{Part 2:} We now establish the converse direction. Let $\mu, \nu \in \mathcal{P}(\Omega_0 \times \mathbb{Y}^1 \times \dots \times \mathbb{Y}^n)$, and assume $\mu$ is more informative than $\nu$. Since this means the expected cost is lower under $\mu$ than under $\nu$ for all games, in particular it is true for games where $\mathbb{U}^i = \mathbb{Y}^i$. We will fix $\mathbb{U}^i = \mathbb{Y}^i$ and evaluate this class of games for the remainder of the proof.

Let $K_Z \mu$ denote the space of all garblings of $\mu$ with common randomness allowed. I.e. 
\[K_Z\mu = \{\nu \in \mathcal{P}(\Omega_0 \times \mathbb{Y}^1 \times \dots \times \mathbb{Y}^n) : y^i_{\nu} = h^i(y^i_\mu, Z)\}\]
for some independent randomness $Z$ that may be common between players. 

We assume that $\nu \in K_Z\mu$ and proceed with a proof by contradiction.

Note that since $L_C(\mu)$ is closed under Assumption \ref{AbsoluteContIndRedG}, via Theorem \ref{LCcompact}. Furthermore, this set is convex, and since $\mathbb{Y}^i = \mathbb{U}^i$ here, we have that $K_Z\mu$ will be closed and convex. This follows from noticing that $L_C(\mu)$ is the set of all measures on $(\Omega \times \prod_{i=1}^n (\mathbb{Y}^i \times \mathbb{U}^i))$ with fixed marginal $\mu$ on $(\Omega \times \prod_{i=1}^n \mathbb{Y}^i)$, where the measure on the action spaces is created through applying all possible policies (which we can view as garblings) with common randomness; since each player's action space and measurement space are assumed to be the same, this means that $K_Z\mu$ is the projection of $L_C(\mu)$ onto the state and action spaces, and so $K_Z\mu$ is also closed and convex.

We will now use the Hahn-Banach Separation Theorem for Locally Convex Spaces. To apply this, we require local convexity of $\mathcal{P}(\Omega_0 \times \prod_{k=1}^n \mathbb{Y}^k)$, and so we define the locally convex space of probability measures with the following notion of convergence: We say that $P_n \rightarrow P$ if 
\begin{align*} \int f(\omega_0,y^1, \dots, y^n)P_n(d\omega_0,dy^1, \dots, dy^n) \rightarrow  \int f(\omega_0,y^1, \dots, y^n)P(d\omega_0,dy^1, \dots, dy^n)\end{align*}

for every measurable and bounded function which is continuous in $(y^1,\dots, y^n)$ for every $\omega_0$. We note that our measures must still have fixed marginal $\zeta$ on $\Omega_0$.

Since continuous and bounded functions {\it separate} probability measures (in the sense that, if the integrations of two measures with respect to continuous functions are equal, the measures must be equal), we can represent every continuous linear map on our space using the form, following \cite[Theorem 3.10]{rudin1991functional}:
\begin{align*}
    F = \int f(\omega_0,y^1, \dots, y^n) \nu (d\omega_0,dy^1, \dots, dy^n), \nonumber
\end{align*}
for some measurable and bounded function $f(\omega_0,y^1, \dots, y^n)$ continuous for every $\omega_0$. It also follows that, given this notion of convergence, $\mathcal{P}(\Omega_0 \times \prod_{k=1}^n \mathbb{Y}^k)$ is a locally convex space. 

Thus, since $\nu$ is compact and convex and $K_Z\mu$ is closed and convex, following a combination of \cite[Theorem 3.4]{rudin1991functional} and \cite[Theorem 3.10]{rudin1991functional}, we have that there exists a continuous and bounded function $f: \Omega_0 \times \mathbb{Y}^1 \times \dots \times \mathbb{Y}^n \rightarrow \mathbb{R}$ such that there exist constants $D_2 > D_1$ with
\[\langle \nu, f \rangle < D_1, \langle \kappa_Z\mu, f \rangle> D_2\quad \quad \forall \kappa_Z\mu \in K_Z\mu. \]

That is to say, \[\int f(\omega_0,y^1, \dots, y^n) \nu(d\omega_0,dy^1, \dots, dy^n) < \int f(\omega_0,y^1, \dots, y^n) \kappa_Z\mu(d\omega_0,dy^1, \dots, dy^n).\] 

Now, since $\mathbb{Y}^i = \mathbb{U}^i$ we can view $f$ as the cost function for some game, where $\langle \nu, f \rangle$ gives the expected cost for the game with information structure $\nu$ when each player plays the identity policy.

We note that because the inequality holds for all $\kappa_Z\mu \in K_Z\mu$, in particular it holds for the minimum over all such measures, so we have the following:
\begin{align*}
   \langle \tilde{\gamma}\nu, f \rangle &< \inf_{\kappa_Z\mu \in K_Z\mu} \langle \kappa_Z\mu, f \rangle \\
   &= \inf_{\kappa_Z\mu \in K_Z\mu}\int f(\omega_0,y^1, \dots, y^n) \kappa_Z\mu(d\omega_0,dy^1, \dots, dy^n) 
\end{align*}
But, this is equivalent to finding the optimal strategic measure in $L_C(\mu)$ to play. 

Thus, we have that, for the cost function $f$, the performance under the identity policy under information structure $\nu$ is strictly better than the performance under the optimal policy for information structure $\mu$. This contradicts the fact that $\mu$ is more informative than $\nu$. Therefore, we must have that $\nu \in K_Z{\mu}$.
\end{proof}

In addition to the extension of Blackwell's results, we can also extend Le Cam's theory to team problems. 
\begin{theorem}\label{LeCamTeams}
Let $\epsilon > 0$ be fixed, $\delta(\mu, \nu) < \epsilon$ if and only if for any team policy $\bar{\gamma}_\nu$ under $\nu$ and any team problem in $G \in \mathcal{E}$ with a bounded cost function $\|c\|_{\infty} \leq 1$, there exists a team policy $\bar{\gamma}_{\mu}$ under $\mu$ such that
\begin{equation}
     J(G, \mu, \bar{\gamma}_{\mu}) < J(G, \nu, \bar{\gamma}_\nu) + \epsilon. \nonumber
\end{equation}
\end{theorem}
\begin{proof}
We note that, when selecting a team policy $\bar{\gamma}$ under an information structure, we view the selection of the team policy as the selection of both the independent randomness $z \in [0,1]^N$ and the maps $\gamma^i: \mathbb{Y}^i \times [0,1]^N \rightarrow \mathbb{U}^i$ for each player. Since $z$ is independent of the other random variables, we will, with a slight abuse of notation, allow the dependence of $z$ to be absorbed into the $\gamma^i$ terms, and thus denote them as maps from $\mathbb{Y}^i$ to $\mathbb{U}^i$. 

\textbf{Step 1.} First we show the forward direction. If $\delta(\mu, \nu) < \epsilon$, then by definition $\inf_{\kappa}\|\kappa \mu - \nu \|_{TV} < \epsilon$. Let $\gamma_{\nu}$ be an arbitrary team policy under $\nu$. Then we have the following:
\begin{align*}
& J(G, \mu, \gamma_{\nu}) - J(G, \nu, \gamma_{\nu}) \\
    & = \int c(\omega_0, \gamma^1_\nu(y^1), \dots, \gamma^n_\nu(y^n))\nu(d\omega_0, dy^1, \dots, dy^n) \\ & \quad \quad \quad \quad - \int c(\omega_0, \gamma^1_\nu(y^1), \dots, \gamma^n_\nu(y^n))\mu(d\omega_0, dy^1, \dots, dy^n) \\
    & \leq \int c(\omega_0, \gamma^1_\nu(y^1), \dots, \gamma^n_\nu(y^n))\nu(d\omega_0, dy^1, \dots, dy^n) \\ & \quad \quad \quad \quad - \inf_{\kappa} \int c(\omega_0, \gamma^1_\nu(y^1), \dots, \gamma^n_\nu(y^n))\kappa\mu(d\omega_0, dy^1, \dots, dy^n) \\
    &= \inf_{\kappa} (\int c(\omega_0, \gamma^1_\nu(y^1), \dots, \gamma^n_\nu(y^n))\nu(d\omega_0, dy^1, \dots, dy^n) \\ & \quad \quad \quad \quad - \int c(\omega_0, \gamma^1_\nu(y^1), \dots, \gamma^n_\nu(y^n))\kappa\mu(d\omega_0, dy^1, \dots, dy^n)) \\
    &\leq \inf_{\kappa}\|\nu - \kappa\mu\|_{TV} < \epsilon.
\end{align*}
The first inequality follows because the second term can not be made larger through an additional filtration of the information through $\kappa$, since the identity garbling could be selected. The second inequality follows from the definition of total variation distance. 

\textbf{Step 2.} For the converse direction, we assume that for any team policy $\bar{\gamma}_\nu$ under $\nu$ and any game with a bounded cost function $\|c\|_{\infty} \leq 1$, there exists a team policy $\bar{\gamma}_{\mu}$ under $\mu$ such that
\begin{equation}
     J(G, \mu, \bar{\gamma}_{\mu}) < J(G, \nu, \bar{\gamma}_\nu) + \epsilon. \nonumber
\end{equation}
In particular, this is true for every game in which $\mathbb{U}^i = \mathbb{Y}^i$ for all players; we denote this by the subclass $\mathcal{G}_{\mathbb{Y} = \mathbb{U}}$. Since the inequality is true for any arbitrary policy $\bar{\gamma}_\nu$, in particular it is true for the team identity policy $\gamma^{id}_i(y^i) = y^i$, for $i = 1,\dots,n$. Thus, we have that for any game in this class, there exists a policy $\bar{\gamma}_{\mu}$ such that:
\begin{align*}
    J(G, \mu, \bar{\gamma}_{\mu}) - J(G, \nu, \gamma^{id}) < \epsilon.
\end{align*}
This implies that the following equation holds:
\begin{equation}
    \sup_{G \in \mathcal{G}_{Y=U}}(J^*(G, \mu) - J(G, \nu, \gamma^{id})) < \epsilon. \nonumber
\end{equation}
We can then complete the proof by noting the following:
\begin{align*}
& \sup_{G \in \mathcal{G}_{Y=U}}(J^*(G, \mu) - J(G, \nu, \gamma^{id})) \\
&\leq \sup_{G \in \mathcal{G}_{Y=U}}|J^*(G, \mu) - J(G, \nu, \gamma^{id})| \\
    &= \sup_{\|c\|\leq 1}|\inf_{\bar{\gamma}_{\mu}}\int c(\omega_0, \gamma^1_{\mu}(y^1), \dots, \gamma^n_{\mu}(y^n))\mu(d\omega_0, dy^1, \dots, dy^n) \\ & \quad \quad \quad \quad \quad \quad \quad \quad - \int c(\omega_0, y^1, \dots, y^n)\nu(d\omega_0, dy^1, \dots, dy^n)| \\
    & = \sup_{\|c\|\leq 1}|\inf_{\kappa}\int c(\omega_0, y^1, \dots, y^n)\kappa\mu(d\omega_0, dy^1, \dots, dy^n) \\ & \quad \quad \quad \quad \quad \quad \quad \quad - \int c(\omega_0, y^1, \dots, y^n)\nu(d\omega_0, dy^1, \dots, dy^n)| \\
    &\leq \inf_{\kappa} \sup_{\|c\|\leq 1}|\int c(\omega_0, y^1, \dots, y^n)\kappa\mu(d\omega_0, dy^1, \dots, dy^n) \\ & \quad \quad \quad \quad \quad \quad \quad \quad - \int c(\omega_0, y^1, \dots, y^n)\nu(d\omega_0, dy^1, \dots, dy^n)| \\
    &=  \inf_{\kappa}\|\kappa \mu - \nu\|_{TV}
\end{align*}
Since the first term is bounded above by $\epsilon$, we have that $\inf_{\kappa}\|\kappa \mu - \nu\|_{TV} < \epsilon$, completing the proof.
\end{proof}

\subsection{Comparison of Information Structures in Zero-Sum Games involving Teams against Teams}

Now, we will prove an extension of Blackwell's comparison of information structures similar to that in \cite{pkeski2008comparison} and \cite{InformationHY}. 

Let $\tilde{\mathcal{G}}$ denote the set of all team-against-team zero-sum games with for which the cost function is continuous and bounded in players' actions for every state. 

\begin{theorem}\label{BlackwellTaT}
Take fixed $\Omega_0, P$, fixed and compact measurement spaces, and information structures $\nu$ and $\mu$ for which Assumption \ref{ABSContA3} holds. Then $\mu$ is more informative for the maximizing team than $\nu$ over all games in $\tilde{\mathcal{G}}$ if and only if the following hold:

\begin{itemize}
    \item[(i).] The measurement of each player in Team 1 under $\nu$ is given by $\tilde{y}^i_1({\nu}) = h^i(y^i_1(\mu), Z_1)$ for $i = 1, \dots, n$ and some independent randomness $Z_1$ (with possibly dependent coordinate components). Denote this information structure by $\kappa_1\mu$.
    \item[(ii).] The measurement of each player in Team 2 under $\mu$ is given by $\tilde{y}^i_2({\mu}) = h^i(y^i_2(\nu), Z_2)$ for $i = 1, \dots, n$ and some independent randomness $Z_2$ (with possibly dependent coordinate components). Denote this information structure by $\kappa_2\nu$. 
\end{itemize}
That is to say, if and only if $\kappa_1\mu = \kappa_2\nu$ for some $\kappa_1$ and $\kappa_2$.

\end{theorem}

\begin{proof}
\textit{(Sketch)}
The proof follows naturally from the argument in \cite{InformationHY}, where instead of showing the results by holding one player constant and using Blackwell's comparison theorem for single-player decision problems, we can hold one team constant and use Theorem \ref{Blackwellteam}. For this reason, only a sketch of the proof is presented here.

\textbf{Step 1:} The forward direction follows from the forward direction to Theorem \ref{Blackwellteam}. Under an information structure, if team $i$'s information is held constant, and team $-i$'s information is garbled with possibly common randomness, then team $-i$ can not perform better under this garbling than they did under the original information structure. This follows Theorem \ref{Blackwellteam} by absorbing team $i$, who have constant information across both information structures, into the cost function for team $-i$ and treating this as a team problem. Applying this for $i = 1,2$ yields the following relationship for any game $G$ in $\tilde{G}:$
\begin{equation}
  V^*(G,\mu) \leq V^*(G, \kappa_1\mu) = V^*(G, \kappa_2\nu) \leq V^*(G, \nu). \nonumber
\end{equation}
Thus, the equilibrium value is lower under $\mu$ than under $\nu$ for every game, and so $\mu$ is more informative than $\nu$. 

\textbf{Step 2:}
For the converse direction, we apply the converse direction to Theorem $\ref{Blackwellteam}$. We can observe that if $V^*(G, \mu) \leq V^*(G, \nu)$ for all $G \in \tilde{\mathcal{G}}$, then it follows that:
\[V(G, \mu, \bar{\gamma}^1_\mu, \bar{\gamma}^2_\nu) \leq V^*(G, \mu)  \leq V^*(G, \nu) ,\]
where $(\bar{\gamma}^1_\mu, \bar{\gamma}^2_\mu)$ and $(\bar{\gamma}^1_\nu, \bar{\gamma}^2_\nu)$ are the team equilibrium policies under $\mu$ and $\nu$, respectively. This equation holds by perturbing the maximizing team's policy to no longer be its optimal policy under $\mu$ in the first term, decreasing the value of the game. 

Then, by comparing the first and third terms above and fixing Team 2's policy as constant for each game $G$, we can absorb Team 2's effect on Team 1 into the cost function for any game in $\tilde{\mathcal{G}}$. Projecting onto the space of team problems for Team 1 this way, we can span all possible team problems, achieving the inequality for Team 1:
\[V^*(E, \mu) \leq V^*(E, \nu),\]
for all team problems $E \in \mathcal{E}$. 

Thus, we get that $\mu$ is more informative than $\nu$ for Team 1. Applying the converse direction to Theorem \ref{Blackwellteam}, we get that condition $(i)$ of the theorem statement holds. Applying a similar analysis by swapping the teams leads to condition $(ii)$. 
\end{proof}

This theorem gives a general comparison of information structures result for both teams and zero-sum games. When only cost functions where one team's actions are irrelevant to the cost are considered, the result reduces to that of team problems, which were discussed in Theorem \ref{Blackwellteam} and in \cite{LehrerRosenbergShmaya}. When each team is restricted to having only one player, the result reduces to a standard zero-sum games, a case that was explored in \cite{pkeski2008comparison} and \cite{InformationHY}. When both of these restrictions are applied, the result reduces to the original single-player decision problem studied by Blackwell. 


\section{Continuity of Equilibrium Values in Information in Team-against-Team Zero-Sum Games}

Recently in \cite{hogeboom2021continuity}, it was shown that for both zero-sum games as well as for stochastic team problems, equilibrium values are continuous when the information structures are perturbed under the total variation topology. Therefore, it is not surprising that we are able to obtain the following regularity results for zero-sum games involving two teams of competing players.

As noted in Section \ref{InfoStTTG1}, the teams share a common cost function $c: \Omega_0 \times \mathbb{U}^1_1 \times \dots \times \mathbb{U}^{N_1}_1 \times \mathbb{U}^1_2 \times \dots \times \mathbb{U}^{N_2}_2 \rightarrow \mathbb{R}$. Team 1's goal is to minimize the expected cost and Team 2's goal is to maximize it. The information structure for the game $\mu$ is a joint probability measure on $\Omega_0 \times \mathbb{Y}^1_1 \times \dots \times \mathbb{Y}^{N_1}_1 \times \mathbb{Y}^1_2 \times \dots \times \mathbb{Y}^{N_2}_2$. We use $\mu_i$ to denote the marginal of $\mu$ on $\Omega_0 \times \mathbb{Y}^i_i \times \dots \times  \mathbb{Y}^{N_i}_i$ for $i = 1,2$ (i.e. $\mu_i$ captures the information available to team $i$).

\begin{theorem}\label{teamateamtheoremCont}
Suppose that the hypotheses of Theorem \ref{SPExistTvsT} apply so that equilibria exist. Then the value function $J^*(c, \mu)$ is continuous under total variation convergence of information structures provided that perturbations in information satisfy the conditions of existence. 
\end{theorem}

\begin{proof}
Following the proof of \cite[Theorem 3.1]{hogeboom2021continuity} and viewing the common randomness variable as an exogenous integration term, the result follows from identical steps. We note again, as in the proof of Theorem \ref{LeCamTeams}, when selecting a team policy $\bar{\gamma}$ under an information structure, we view the selection of the policy as the selection of both the independent randomness $z \in [0,1]^N$ and the maps $\gamma^i: \mathbb{Y}^i \times [0,1]^N \rightarrow \mathbb{U}^i$ for each player. Since $z$ is independent of the other random variables, we will, with a slight abuse of notation, allow the dependence of $z$ to be absorbed into the $\gamma^i$ terms, and thus denote them as maps from $\mathbb{Y}^i$ to $\mathbb{U}^i$. 

Theorem \ref{SPExistTvsT} shows that an equilibrium exists under $\mu$ and $\nu$. Assume that $\mu$ and $\nu$ are such that $\|\mu - \nu\|_{TV} \leq \epsilon$ for some $\epsilon >0$. Without loss of generality, assume that $J^*(g,\mu) - J^*(g, \nu) \geq 0$ for some game $g$ (a symmetric argument can be applied in the case where $J^*(g,\mu) - J^*(g, \nu) \leq 0$).
Then we have the following:
\begin{align*}
&J^*(g,\mu) - J^*(g, \nu) \\
& = \int c(x, \bar{\gamma}^1_{\mu}, \bar{\gamma}^2_{\mu})\mu(dx; dy^1_1,\cdots,dy^{N_1}_1; dy^2_1,\cdots,dy^{N_2}_2)  \\
&\qquad \qquad \qquad - \int c(x, \bar{\gamma}^1_{\nu}, \bar{\gamma}^2_{\nu})\nu(dx; dy^1_1,\cdots,dy^{N_1}_1; dy^2_1,\cdots,dy^{N_2}_2)  \\
&\leq \int c(x, \bar{\gamma}^1_{\nu}(y^1), \bar{\gamma}^2_{\mu}(y^2))\mu(dx; dy^1_1,\cdots,dy^{N_1}_1; dy^2_1,\cdots,dy^{N_2}_2)  \\
&\qquad \qquad \qquad - \int c(x, \bar{\gamma}^1_{\nu}(y^1), \bar{\gamma}^2_{\mu}(y^2))\nu(dx; dy^1_1,\cdots,dy^{N_1}_1; dy^2_1,\cdots,dy^{N_2}_2).
\end{align*}

The inequality comes from perturbing the minimizer team's equilibrium strategy in the first term, making the expected cost larger, and perturbing the maximizer team's equilibrium strategy in the second term, making the expected cost smaller. Since the functions to be integrated are identical and by considering the symmetric argument for the case with $J^*(g,\mu) - J^*(g, \nu) \leq 0$, we have that $\|\mu_n - \mu\|_{TV} \rightarrow 0$, then $|J^*(g,\mu_n) - J^*(g,\mu)| \rightarrow 0$. 


\end{proof}



\section{Conclusion}
For games with Standard Borel measurement and action spaces involving two teams of finitely many agents taking part in a zero-sum game, we presented an existence result for saddle-point equilibria when common randomness is assumed to be available in each team with an analysis on conditions for compactness of strategic team measures provided. Then, Blackwell's ordering of information structures was generalized to $n$-player teams with standard Borel spaces, where correlated garbling of information structures is introduced as a key attribute, and building on this result Blackwell's ordering of information structures was defined for team-against-team zero-sum game problems. Finally, continuity of equilibrium value of team-against-team zero-sum game problems in the space of information structures under total variation was established.

\bibliographystyle{plain}
\bibliography{IanBib,SerdarBibliography}

\end{document}